\documentclass[a4paper,10pt]{article}

\usepackage{amsfonts}
\usepackage{amsthm}
\usepackage{amssymb}
\usepackage{graphicx, enumerate}
\usepackage{amsmath}
\usepackage{latexsym}
\usepackage{longtable}
\usepackage{tabularx}
\usepackage{amsmath}
\usepackage{amsfonts}
\usepackage{amsthm}
\usepackage{setspace}
\usepackage{graphicx}
\newtheorem{thm}{Theorem}[section]

\newtheorem{lemma}[thm]{Lemma}

\newtheorem{cor}[thm]{Corollary}

\newtheorem{rem}[thm]{Remark}

\input amssym.def
\input amssym.tex

\newcommand{\be}{\begin{equation}}
\newcommand{\ee}{\end{equation}}
\newcommand{\bea}{\begin{eqnarray}}
\newcommand{\eea}{\end{eqnarray}}
\newcommand{\bean}{\begin{eqnarray*}}
\newcommand{\eean}{\end{eqnarray*}}
\def\non{\nonumber}

\long\def\delete#1{}

\usepackage{color}

\definecolor{VeryLightBlue}{rgb}{0.9,0.9,1}
\definecolor{LightBlue}{rgb}{0.8,0.8,1}
\definecolor{MidBlue}{rgb}{0.5,0.5,1}
\definecolor{DarkBlue}{rgb}{0,0,0.6}
\definecolor{Blue}{rgb}{0,0,1}

\definecolor{Gold}{rgb}{1,0.843,0}

\definecolor{LightGreen}{rgb}{0.88,1,0.88}
\definecolor{MidGreen}{rgb}{0.6,1,0.6}
\definecolor{DarkGreen}{rgb}{0,0.6,0}

\definecolor{VeryLightYellow}{rgb}{1,1,0.9}
\definecolor{LightYellow}{rgb}{1,1,0.6}
\definecolor{MidYellow}{rgb}{1,1,0.5}
\definecolor{DarkYellow}{rgb}{1,1,0.2}

\definecolor{DarkPurple}{rgb}{.6,0,1}

\definecolor{Red}{rgb}{1,0,0}
\definecolor{VeryLightRed}{rgb}{1,0.9,0.9}
\definecolor{LightRed}{rgb}{1,0.8,0.8}
\definecolor{MidRed}{rgb}{1,0.55,0.55}

\def\Cay{{\rm Cay}}

\def\s{\sigma}

\def\Z{\mathbb{Z}}

\begin{document}

\title
{\Large \sc \bf Group distance magic and antimagic graphs}
\date{}
\author{S. Cichacz\\
{\small AGH University of Science and Technology, Faculty of Applied Mathematics,}\\
{\small Al. Mickiewicza 30, 30-059 Krak\'{o}w, Poland}\\
{\small Email: cichacz@agh.edu.pl}\\
D. Froncek\\
{\small Department of Mathematics and Statistics, University of Minnesota Duluth,}\\ 
{\small 1117 University Dr., Duluth, MN 55812-3000, U.S.A.}\\
{\small Email: dalibor@d.umn.edu} \\
K. Sugeng\\
{\small Department of Mathematics, Faculty of Mathematics and Natural Sciences,}\\
{\small University of Indonesia, Kampus UI Depok, Depok 16424, Indonesia}\\
{\small Email: kiki@sci.ui.ac.id}\\
Sanming Zhou\thanks{Corresponding author.}\\
{\small School of Mathematics and Statistics, The
University of Melbourne,}\\ 
{\small Parkville, VIC 3010, Australia}\\
{\small Email: sanming@unimelb.edu.au}}

\openup 0.5\jot
\maketitle
\date

\begin{abstract}
Given a graph $G$ with $n$ vertices and an Abelian group $A$ of order $n$, an {$A$-distance antimagic labelling} of $G$ is a bijection from $V(G)$ to $A$ such that the vertices of $G$ have pairwise distinct weights, where the weight of a vertex is the sum (under the operation of $A$) of the labels assigned to its neighbours. An {$A$-distance magic labelling} of $G$ is a bijection from $V(G)$ to $A$ such that the weights of all vertices of $G$ are equal to the same element of $A$. In this paper we study these new labellings under a general setting with a focus on product graphs. We prove among other things several general results on group antimagic or magic labellings for Cartesian, direct and strong products of graphs. As applications we obtain several families of graphs admitting group distance antimagic or magic labellings with respect to elementary Abelian groups, cyclic groups or direct products of such groups.


\end{abstract}

\section{Introduction}
 
A {\em distance magic labelling} of a graph $G = (V(G), E(G))$ with $|V(G)|=n$ is a bijection from $V(G)$ to $\{1,2,\dots,n\}$ such that the weights of all vertices are equal to the same constant called the {\em magic constant}, where the {\em weight} of a vertex is the sum of the labels of its neighbours. A {\em distance antimagic labelling} of $G$ is a bijection from $V(G)$ to $\{1,2,\dots,n\}$ such that different vertices have distinct weights. A graph that admits a distance magic (antimagic) labelling is called a {\em distance magic (antimagic) graph}. Motivated by the construction of magic squares, distance magic labelling received considerable attention in recent years. The reader is referred to \cite{afk} for a survey on distance magic labelling, \cite{BM,wal} for various kinds of magic and antimagic labelling problems for graphs, and \cite{gal} for a dynamic survey of graph labellings.

The purpose of the present paper is to study distance magic and antimagic labelling by using elements of Abelian groups. Given a graph $G$ with $n$ vertices and an Abelian group $A$ of order $n$, we define an {\em $A$-distance magic labelling} of $G$ to be a bijection $f$ from $V(G)$ to $A$ such that the weight of every vertex $x \in V(G)$ is equal to the same element of $A$, called the {\em magic constant} of $G$ with respect to $f$. Here the \emph{weight} of $x$ under $f$ is defined as
$$
w_f(x) = \sum_{y \in N_G(x)} f(y)
$$
with the understanding that the addition is performed in the group $A$, where $N_G(x)$ is the (open) neighbourhood of $x$ in $G$. (We usually write $w(x)$ in place of $w_f(x)$ and $N(x)$ instead of $N_G(x)$ if there is no risk of confusion.) We define an \emph{$A$-distance antimagic labelling} of $G$ to be a bijection $f$ from $V(G)$ to $A$ such that the weights of all vertices of $G$ under $f$ are pairwise distinct. 
Obviously, every distance magic graph with $n$ vertices admits a $\Z_n$-distance magic labelling, but the converse is not necessarily true. It is also clear that under any $A$-distance antimagic labelling every element of $A$ occurs as the weight of precisely one vertex of $G$. 

In the special case when the group involved is cyclic, the concept of group distance magic labelling was introduced by the second author \cite{Fr}. The present paper represents the first attempt towards a systematic study of group distance magic and group distance antimagic labellings of graphs under a general setting. The theme for these new labellings is the following question for various families of graph-group pairs $(G, A)$ with $|V(G)| = |A|$: Under what conditions does $G$ admit an $A$-distance antimagic/magic labelling?
  
The main results in this paper are as follows. In Section \ref{sec:nec-suff} we give a necessary condition for a graph with an even number of vertices to be distance antimagic with respect to an Abelian group with a unique involution (Theorem \ref{thm:r-n-even}). We then give sufficient conditions for a Cayley graph on an Abelian group to be distance antimagic or magic with respect to the same group (Theorem \ref{obs:cay}), and discuss the consequences of these results to Cayley graphs on elementary Abelian groups.

In Section \ref{sec:cart} we study group distance antimagic and magic labellings of Cartesian products of graphs. Among other things we prove two general results, Theorems \ref{thm:cart-prod-gen} and \ref{thm:cart-prod-gen1}, which provide machineries for constructing new group distance antimagic or magic graphs. These can be applied to many special cases, and we illustrate this by a few corollaries with a focus on hypercubes, Hamming graphs and Cartesian products of cycles.

In Section \ref{sec:dir} we prove a general result (Theorem \ref{gr_pr}) that can be used to construct group distance antimagic graphs by means of direct products of graphs. We then give a few families of such graphs in four corollaries.

In Section \ref{sec:Zn} we focus on distance antimagic graphs with respect to cyclic groups $\Z_n$. We obtain several general results (Theorems \ref{thm:bal}, \ref{thm:dir}, \ref{thm:str} and \ref{Knregular}) for Cartesian, direct and strong products, and we then use them to give concrete families of $\Z_n$-distance antimagic groups in a few corollaries. In Section \ref{sec:circulant} we give two sufficient conditions for a circulant graph with $n$ vertices to be antimagic with respect to $\Z_n$. We conclude the paper with remarks and open problems.

We notice that graph labelling by Abelian groups has been studied in the literature under different settings. For example, in \cite{ComNelPal} vertex-magic and edge-magic total labellings of graphs by Abelian groups were studied. In \cite{CCN}, edge-magic total-labellings of countably infinite graphs by $\Z$ were investigated. In \cite{LL, LS, SLS}, vertex-magic edge-labellings of graphs by Abelian groups were studied.  

We refer the reader to \cite{Scott} for group-theoretic terminology. We use $A_1 \times \cdots \times A_d$ to denote the direct product of groups $A_1, \ldots, A_d$. In particular, $\Z_{q}^{d} = \Z_q \times \cdots \times \Z_q$ ($d$ factors), and when $q=p$ is a prime $\Z_{p}^d$ is an elementary Abelian $p$-group. As usual we write the operation of an Abelian group $A$ additively, denote its identity element by $0$, and write the inverse element of $x \in A$ by $-x$. All groups in the paper are finite and Abelian, and all graphs considered are finite, simple and undirected. As usual the cardinality of a set $X$ is denoted by $|X|$.

\section{A necessary condition, and sufficient conditions for Cayley graphs}
\label{sec:nec-suff}

A well-known result \cite{Scott} in group theory asserts that a finite group contains involutions (that is, elements with order 2) if and only if it is of even order. Denote $\s(A) = \sum_{x \in A} x$. Then $\s(A)$ is equal to the sum of the involutions of $A$. The following lemma will be used in the proof of Theorem \ref{thm:r-n-even}.

\begin{lemma}(\cite[Lemma 8]{ComNelPal})
\label{lem:inv} Let $A$  be a finite Abelian group.
\begin{itemize}
  \item[\rm (a)] If $A$ has exactly one involution, say, $a$, then $\s(A)= a$.
  \item[\rm (b)] If $A$ has no involutions, or more than one involution, then $\s(A)=0$.
\end{itemize}
\end{lemma}

\begin{thm}
\label{thm:r-n-even}
Let $G$ be an $r$-regular graph on $n$ vertices, where $n$ is even. Then for any Abelian group $A$ of order $n$ with exactly one involution, $G$ cannot be $A$-distance antimagic unless $r$ is odd.
\end{thm}

\begin{proof}
Let $A$ be an Abelian group of order $n$ with exactly one involution, say, $a$.
Suppose $r$ is even and $G$ is $A$-distance antimagic with $f$ the desired labelling. Denote $w(G)=\sum_{x \in V(G)}w(x)$. Since $f$ is an $A$-antimagic labelling, we have $w(G)=\s(A)$. On the other hand, since $G$ is $r$-regular, $w(G)=\sum_{x \in V(G)}\left(\sum_{y \in N(x)}f(y)\right)=r\s(A)$. Hence $r\s(A) = \s(A)$. Since $r$ is even and by Lemma \ref{lem:inv} $\s(A)=a$ is the involution, we have $r\s(A) = 0$, which implies $\s(A) = 0$, a contradiction.
\end{proof}

Since the cyclic group $\Z_n$ has exactly one involution when $n$ is even, Theorem~\ref{thm:r-n-even} implies the following result.

\begin{cor}
\label{obs:zn-even}
Let $G$ be an $r$-regular graph on $n$ vertices such that both $n$ and $r$ are even. Then $G$ is not
 $\mathbb{Z}_n$-distance antimagic.
\end{cor}

The fundamental theorem of finite Abelian groups \cite{Scott} asserts that any finite Abelian group $A$ can be decomposed into the direct product of cyclic subgroups of prime power orders, and this decomposition is unique up to ordering of the cyclic subgroups. It can be verified that, if in this decomposition there are exactly $t$ cyclic subgroups with order a power of $2$, then $A$ has exactly $2^t-1$ involutions. In particular, if $|A| = n \equiv 2 \pmod 4$, then $A$ has exactly one involution. Thus, by Theorem ~\ref{thm:r-n-even}, we obtain:

\begin{cor}
\label{obs:n2mod4}
Let $G$ be an $r$-regular graph on $n$ vertices such that $n \equiv 2 \pmod 4$ and $r$ is even. There does not exists an Abelian group $A$ of order $n$ such that $G$ is $A$-distance antimagic.
\end{cor}

The \emph{exponent} \cite{Scott} of a finite Abelian group $A$, denoted $\exp(A)$, is the least positive integer $m$ such that $m x = 0$ for every $x \in A$. In particular, $\exp(\Z_{n_1} \times \cdots \times \Z_{n_k})$ is equal to the least common multiple of $n_1, \ldots, n_k$.

Given an Abelian group $A$ and a subset $S \subseteq A \setminus \{0\}$ such that $-S = S$, where $-S = \{-s: s \in S\}$, the \emph{Cayley graph} $\Cay(A, S)$ of $A$ with respect to the {\em connection set} $S$ is defined to have vertex set $A$ such that $x, y \in A$ are adjacent if and only if $x - y \in S$. It is evident that $\Cay(A, S)$ is an $|S|$-regular graph.

\begin{thm}\label{obs:cay}
Let $A$ be a finite Abelian group of order $n = |A|$, and $G = \Cay(A, S)$ a Cayley graph on $A$ of degree $r = |S|$.
\begin{itemize}
\item[\rm (a)] If $n$ and $r$ are coprime, then $G$ is $A$-distance antimagic, and any automorphism of $A$ is an $A$-distance antimagic labelling of $G$.
\item[\rm (b)] If $\exp(A)$ is a divisor of $r$, then $G$ is $A$-distance magic, and any automorphism $f$ of $A$ is an $A$-distance magic labelling of $G$ with magic constant $\sum_{s \in S} f(s)$.
\end{itemize}
\end{thm}

\begin{proof}
Let $f$ be an automorphism of $A$. Since $A$ is Abelian and the neighbourhood in $G$ of each $x \in A$ is $\{x+s: s \in S\}$, the weight of $x \in A$ under $f$ is given by
$$
w(x) = \sum_{s \in S} f(x+s) = \sum_{s \in S} (f(x) + f(s)) = r f(x) + \sum_{s \in S} f(s).
$$

(a) Suppose $n$ and $r$ are coprime. If $w(x) = w(y)$ for distinct $x, y \in A$, then $r f(x) = r f(y)$ and hence $r (f(x) - f(y)) = 0$. Since $f(x) \ne f(y)$, this implies that the order $o(f(x) - f(y))$ of $f(x) - f(y)$ in $A$, which is greater than 1 as $f(x) - f(y) \ne 0$, is a divisor of $r$.
 On the other hand, $o(f(x) - f(y))$ is a divisor of $n$. Thus $o(f(x) - f(y))$ is a common divisor of $n$ and $r$ that is greater than 1, contradicting our assumption $\gcd(n, r) = 1$. Therefore, $w(x) \ne w(y)$ for distinct $x, y \in A$ and so $f$ is an $A$-distance antimagic labelling of $G$.

(b) Suppose $\exp(A)$ divides $r$. Then for distinct $x, y \in A$, $o(f(x) - f(y))$ is a divisor of $r$ and as such $r (f(x) - f(y)) = 0$. Thus the computation above yields $w(x) = w(y)$. Therefore, $f$ is an $A$-distance magic labelling of $G$ with magic constant $\sum_{s \in S} f(s)$.
\end{proof}

In part (b) of Theorem \ref{obs:cay}, the magic constant $\sum_{s \in S} f(s)$ relies on the automorphism $f$ of $A$. This shows that a graph may have several magic labellings with respect to the same group but with distinct magic constants.
Since $\exp(\Z_{p}^d) = p$, Theorem \ref{obs:cay} implies the following corollary.

\begin{cor}
\label{cor:cay}
Let $p \ge 2$ be a prime, and let $r \ge 2$ and $d \ge 1$ be integers.
\begin{itemize}
\item[\rm (a)] If $p$ is not a divisor of $r$, then any Cayley graph on $\Z_{p}^d$ with degree $r$ is $\Z_{p}^d$-distance antimagic, and any automorphism of $\Z_{p}^d$ is a $\Z_{p}^d$-distance antimagic labelling.
\item[\rm (b)] If $p$ is a divisor of $r$, then any Cayley graph $\Cay(\Z_{p}^d, S)$ on $\Z_{p}^d$ with degree $r$ is $\Z_{p}^d$-distance magic, and any automorphism $f$ of $\Z_{p}^d$ is a $\Z_{p}^d$-distance magic labelling with magic constant $\sum_{s \in S} f(s)$.
\end{itemize}
\end{cor}

In particular, for the elementary Abelian 2-groups $\Z_{2}^d$, Corollary \ref{cor:cay} yields the following result.

\begin{cor}
\label{cor:cay1}
Let $d \ge 2$ be an integer.
\begin{itemize}
\item[\rm (a)] Any Cayley graph on $\Z_{2}^d$ with an odd degree is $\Z_{2}^d$-distance antimagic, with any automorphism of $\Z_{2}^d$ as a $\Z_{2}^d$-distance antimagic labelling.
\item[\rm (b)] Any Cayley graph $\Cay(\Z_{2}^d, S)$ on $\Z_{2}^d$ with an even degree $|S|$ is $\Z_{2}^d$-distance magic, with any automorphism $f$ of $\Z_{2}^d$ as a $\Z_{2}^d$-distance magic labelling with magic constant $\sum_{s \in S} f(s)$.
\end{itemize}
\end{cor}

Since the \emph{hypercube} $Q_d$ is the Cayley graph $\Cay(\Z_2^d, S)$ with $S = \{(1, 0, \ldots, 0), \ldots, (0, 0, \ldots, 1)\}$, Corollary \ref{cor:cay1} implies:

\begin{cor}
\label{thm:mag-cube-Z_2^d}
Let $d \ge 2$ be an integer.
\begin{itemize}
\item[\rm (a)] If $d$ is odd, then $Q_d$ is $\Z_{2}^d$-distance antimagic.
\item[\rm (b)] If $d$ is even, then $Q_d$ is $\Z_{2}^d$-distance magic with magic constant $(1, \ldots, 1)$ (\cite{CiFrGrKo}).  
\end{itemize}
\end{cor}

As we will see in the next section, both parts of this corollary are special cases of some more general results, namely Corollary \ref{cor:cart-prod-1} and Theorem \ref{thm:mag-ham}, respectively. In Theorem \ref{thm:hypercube} and Remark \ref{rem:Q6} we will give more sufficient conditions for $Q_d$ to be $\Z_{2}^d$-distance antimagic. The magic constant $(1, \ldots, 1)$ in (b) above is obtained from the trivial automorphism of $\Z_{2}^d$. Different choices of the automorphism $f$ in (b) of Corollary \ref{cor:cay1} may result in different magic constants for $Q_d$.

\section{Labelling Cartesian products}
\label{sec:cart}

The {\em Cartesian product} $G_1 \Box \cdots \Box G_d$ of $d$ graphs $G_1, \ldots, G_d$ is defined \cite{IK} to be the graph with vertex set $V(G_1) \times \cdots \times V(G_d)$ such that $(x_1, \ldots, x_d)$ and $(y_1, \ldots, y_d)$ are adjacent if and only if $x_{i} \neq y_{i}$ for exactly one $i$, and for this $i$, $x_{i}$ and $y_{i}$ are adjacent in $G_i$. The Cartesian product $H_{q_1, \ldots, q_d} = K_{q_1} \Box \cdots \Box K_{q_d}$
of complete graphs is called a {\em Hamming graph}, where $q_1, \ldots, q_d \ge 2$ are integers. Equivalently, $H_{q_1, \ldots, q_d}$ is the Cayley graph on $\Z_{q_1} \times \cdots \times \Z_{q_d}$ with respect to the connection set $S = \{(s_1, \ldots, s_d): s_i \in \Z_{q_i}, 1 \le i \le d,\;\mbox{there exists exactly one $i$ such that $s_i \neq 0$}\}$. If $q_1 = \cdots = q_d = q$, we write $H(d, q)$ in place of $H_{q, \ldots, q}$. Thus $H(d, q)$ is the Cayley graph on $\Z_{q}^d$ in which $(x_1, \ldots, x_d)$ and $(y_1, \ldots, y_d)$ are adjacent if and only if they differ at exactly one coordinate.  

\begin{thm}
\label{ham}
Let $d \ge 1$ and $q \ge 2$ be integers.
\begin{itemize}
\item[\rm (a)] If $d$ and $q$ are coprime, then $H(d,q)$ is $\Z_{q^d}$-distance antimagic. In particular, $K_q$ is $\Z_{q}$-distance antimagic for any $q \ge 2$.
\item[\rm (b)] If both $d$ and $q$ are even, then $H(d,q)$ is not $\Z_{q^d}$-distance antimagic.
\end{itemize}
\end{thm}

\begin{proof}
(a) Define
$$
f(x) = \sum_{i=1}^d x_i q^{i-1},\; x = (x_1, \ldots, x_d),\; x_i \in \{0, 1, \ldots, q-1\}.
$$
Then $0 \le f(x) \le q^d - 1$ and $f$ is a bijection from $\Z_{q}^d$ to $\Z_{q^d}$.
The $d(q-1)$ neighbours of $x$ are $y_{j, t} = (x_1, \ldots, x_{j-1}, t, x_{j+1}, \ldots, x_d)$, $1 \le j \le d$, $t \in \Z_{q} \setminus \{x_j\}$. Note that the label of $y_{j, t}$ is $f(y_{j, t}) = f(x) - x_j q^{j-1} + t q^{j-1}$. Thus the weight of $x$ under $f$ is given by
\bean
w(x) & = & d(q-1)f(x) + \sum_{j=1}^d \left(\sum_{t \ne x_j} (t - x_j) q^{j-1}\right)\\
        & = & d(q-1)f(x) + \sum_{j=1}^d \left(\sum_{t=0}^{q-1} (t - x_j) q^{j-1}\right)\\
        & = & d(q-1)f(x) + \sum_{j=1}^d \left(\frac{q(q-1)}{2} \cdot q^{j-1} - x_j q^{j}\right)\\
        & = & d(q-1)f(x) + \frac{q(q-1)}{2} \cdot \frac{q^{d}-1}{q-1} - qf(x)\\
        & = & \left(d(q-1)-q\right)f(x) + \frac{q(q^d-1)}{2}.\\
\eean
Since $\gcd(d, q) = 1$ by our assumption, we have $\gcd(d(q-1)-q, q^d) = 1$. Since $f$ is a bijection from $\Z_{q}^d$ to $\Z_{q^d}$, it follows that for different $x, x' \in \Z_{q}^d$ we have $w(x) \ne w(x')$. In other words, $f$ is a $\Z_{q^d}$-distance antimagic labelling of $H(d, q)$.

In particular, since $K_q \cong H(1,d)$ and $\gcd(1, q) = 1$, it follows from what we proved above that $K_q$ is $\Z_{q}$-distance antimagic for any $q \ge 2$.

(b) Since both $d$ and $q$ are even, the degree $d(q-1)$ and the order $q^d$ of $H(d, q)$ are even. Thus, by Corollary \ref{obs:zn-even}, $H(d,q)$ is not $\Z_{q^d}$-distance antimagic.
\end{proof}

Since $Q_d = H(d, 2)$, Theorem \ref{ham} implies the following result.

\begin{cor}
\label{thm:cube-Zn}
The $d$-dimensional hypercube $Q_d$ is $\Z_{2^d}$-distance antimagic if and only if $d$ is odd.
\end{cor}

\begin{thm}
\label{thm:mag-ham}
Suppose $d, q \ge 2$ are integers such that $q$ is a divisor of $d$. Then $H(d, q)$ is $\Z_{q}^{d}$-distance magic, with magic constant $\left(\frac{q}{2}, \ldots, \frac{q}{2}\right)$ when $q$ is even and $(0, \ldots, 0)$ when $q$ is odd.
\end{thm}

\begin{proof}
Define
$$
f(x) = (x_1, \ldots, x_d),\; x = (x_1, \ldots, x_d) \in \Z_q^d.
$$
Then $f$ is a bijection from the vertex set $\Z_q^d$ of $H(d, q)$ to the group $\Z_q^d$. With operations modulo $q$ the weight of $x$ under $f$ is given by
\bean
w(x) & = & \sum_{i = 1}^d \sum_{t \in \Z_q \setminus \{x_i\}} (x_1, \ldots, x_{i-1}, t, x_{i+1}, \ldots, x_d) \\
& = & \sum_{i = 1}^d \left((q-1)x_1, \ldots, \frac{q(q-1)}{2} - x_i, \ldots, (q-1)x_d\right) \\
& = & \sum_{i = 1}^d \left\{(q-1)(x_1, \ldots, x_d)+\left(0, \ldots, \frac{q(q-1)}{2} - qx_i, \ldots, 0\right)\right\}\\
& = & \sum_{i = 1}^d \left\{-(x_1, \ldots, x_d)+\left(0, \ldots, \underbrace{\frac{q(q-1)}{2}}_{i-th}, \ldots, 0\right)\right\} \\
& = & -d(x_1, \ldots, x_d) + \frac{q(q-1)}{2} (1, \ldots, 1).
\eean

Since $q$ is a divisor of $d$ by our assumption, we have $d(x_1, \ldots, x_d) = (0, \ldots, 0)$ in $\Z_{q}^{d}$. Therefore, $w(x) = \frac{q(q-1)}{2} (1, \ldots, 1)$ for every vertex $x \in \Z_q^d$ of $H(d, q)$, and hence $H(d, q)$ is $\Z_{q}^{d}$-distance magic. If $q$ is even, then the magic constant is $\frac{q(q-1)}{2} (1, \ldots, 1) = \frac{q}{2}(q-1)(1, \ldots, 1) = -\frac{q}{2}(1, \ldots, 1) = \frac{q}{2}(1, \ldots, 1)$. If $q$ is odd, then the magic constant is $\frac{q-1}{2} q (1, \ldots, 1) = (0, \ldots, 0)$.
\end{proof}

Since $H(d, q)$ has degree $d(q-1)$, in the special case when $q=p$ is a prime factor of $d$, the fact that $H(d, p)$ is $\Z_{p}^d$-distance magic is also implied by part (b) of Corollary \ref{cor:cay}.

In the special case when $q=2$, Theorem \ref{thm:mag-ham} gives (b) of Corollary \ref{thm:mag-cube-Z_2^d}.

\begin{thm}
\label{thm:cart-prod-gen}
Let $G_i$ be an $r_i$-regular graph with $n_i \ge 2$ vertices, $1 \le i \le k$. Let $A_i$ be an Abelian group of order $n_i$ such that $\exp(A_i)$ is a divisor of $r - r_i$, $1 \le i \le k$, where $r = \sum_{i=1}^k r_i$.
\begin{itemize}
\item[\rm (a)] If $G_i$ is $A_i$-distance antimagic for $1 \le i \le k$, then $G_1 \Box \cdots \Box G_k$ is $A_1 \times \cdots \times A_k$-distance antimagic.
\item[\rm (b)] If $G_i$ is $A_i$-distance magic for $1 \le i \le k$, then $G_1 \Box \cdots \Box G_k$ is $A_1 \times \cdots \times A_k$-distance magic.
\end{itemize}
\end{thm}

\begin{proof}
(a) Since $G_i$ is $A_i$-distance antimagic, it admits an $A_i$-distance antimagic labelling, say, $g_i: V(G_i) \rightarrow A_i$, $1 \le i \le k$. Define  
\be
\label{eq:f}
f(x_1, \ldots, x_k) = (g_1(x_1), \ldots, g_k(x_k)),\; x_i \in V(G_i), 1 \le i \le k.
\ee
Then $f$ is a bijection from $V(G_1) \times \cdots \times V(G_k)$ to $A_1 \times \cdots \times A_k$. Denote by $w_{G_i}(x_i)$ the weight of $x_i$ under $g_i$. Then the weight of $x = (x_1, \ldots, x_k)$ under $f$ is given by
\bea
w(x) & = & \sum_{i=1}^k\; \sum_{x_i':\, x_i x_i' \in E(G_i)} f(x_1, \ldots, x_{i-1}, x'_i, x_{i+1}, \ldots, x_k)\non \\
& = & \sum_{i=1}^k\; \sum_{x_i':\, x_i x_i' \in E(G_i)} (g_1(x_1), \ldots, g_{i-1}(x_{i-1}), g_i(x'_i), g_{i+1}(x_{i+1}), \ldots, g_k(x_k))\non \\
& = & \sum_{i=1}^k (r_i g_1(x_1), \ldots, r_i g_{i-1}(x_{i-1}), w_{G_i}(x_i), r_ig_{i+1}(x_{i+1}), \ldots, r_i g_k(x_k))\non \\
& = & r(g_1(x_1), \ldots, g_k(x_k)) + \sum_{i=1}^k (0, \ldots, 0, w_{G_i}(x_i) - r_ig_{i}(x_{i}), 0, \ldots, 0)\non \\
& = & (w_{G_1}(x_1) + (r-r_1)g_{1}(x_1), \ldots, w_{G_k}(x_k) + (r-r_k)g_{k}(x_k)). \label{eq:w}
\eea
Since $\exp(A_i)$ is a divisor of $r-r_i$ by our assumption, we have $(r-r_i)g_{i}(x_i) = 0$ in $A_i$ for each $i$. Hence $w(x) = (w_{G_1}(x_1), \ldots, w_{G_k}(x_k))$. (In all computations, the operation on the $i$th coordinate is performed in $A_i$.) Since each $G_i$ is assumed to be $A_i$-distance antimagic, it follows that $f$ is an $A_1 \times \cdots \times A_k$-distance antimagic labelling of $G_1 \Box \cdots \Box G_k$.

(b) Define $f$ as in (a) with the understanding that each $g_i$ is an $A_i$-distance magic labelling of $G_i$. The computation above yields $w(x) = (w_{G_1}(x_1), \ldots, w_{G_k}(x_k))$ for every $x$.
Since $g_i$ is an $A_i$-distance magic labelling of $G_i$, say, with magic constant $\mu_i \in A_i$, we have $w_{G_i}(x_i) = \mu_i$ for all $x_i \in V(G_i)$. Therefore, $w(x) = (\mu_1, \ldots, \mu_k)$ for all $x$, and hence $G_1 \Box \cdots \Box G_k$ is $A_1 \times \cdots \times A_k$-distance magic.
\end{proof}

Since $\exp(\Z_{n}^{d}) = n$, Theorem \ref{thm:cart-prod-gen} implies:

\begin{cor}
\label{thm:cart-prod}
Let $d_i, r_i \ge 1$ and $n_i \ge 2$ be integers, $1 \le i \le k$. Let $G_i$ be an $r_i$-regular graph with $n_i^{d_i}$ vertices, $1 \le i \le k$. Suppose $n_i$ is a divisor of $r - r_i$, $1 \le i \le k$, where $r = \sum_{i=1}^k r_i$.
\begin{itemize}
\item[\rm (a)] If $G_i$ is $\Z_{n_i}^{d_i}$-distance antimagic for $1 \le i \le k$, then $G_1 \Box \cdots \Box G_k$ is $\Z_{n_1}^{d_1} \times \cdots \times \Z_{n_k}^{d_k}$-distance antimagic.
\item[\rm (b)] If $G_i$ is $\Z_{n_i}^{d_i}$-distance magic for $1 \le i \le k$, then $G_1 \Box \cdots \Box G_k$ is $\Z_{n_1}^{d_1} \times \cdots \times \Z_{n_k}^{d_k}$-distance magic.
\end{itemize}
\end{cor}

Setting $k = 2$ and $n_1 = n_2 = 2$ in Corollary \ref{thm:cart-prod}, we obtain:

\begin{cor}
\label{cor:cart-prod-2}
Let $G, H$ be regular graphs with $2^d, 2^e$ vertices, respectively. Suppose both $G$ and $H$ have even degrees.
\begin{itemize}
\item[\rm (a)] If $G$ is $\Z_2^d$-distance antimagic and $H$ is $\Z_2^e$-distance antimagic, then $G \Box H$ is $\Z_2^{d+e}$-distance antimagic.
\item[\rm (b)] If $G$ is $\Z_2^d$-distance magic and $H$ is $\Z_2^e$-distance magic, then $G \Box H$ is $\Z_2^{d+e}$-distance magic.
\end{itemize}
\end{cor}

Theorem \ref{thm:cart-prod-gen} and its corollaries above enable us to construct group distance antimagic/magic graphs from known ones. As an example, from Corollaries \ref{cor:cay} and \ref{thm:cart-prod} we obtain the following result (a general result involving more than two factors can be formulated similarly).

\begin{cor}
\label{cor:two}
Let $p_i \ge 2$ be a prime and $d_i, r_i \ge 1$ be integers with $2 \le r_i \le p_i^{d_i}$ for $i = 1, 2$. Let $G_i$ be any Cayley graph on $\Z_{p_i}^{d_i}$ with degree $r_i$ for $i = 1, 2$.
\begin{itemize}
\item[\rm (a)] If $p_1$ divides $r_2$ but not $r_1$, and $p_2$ divides $r_1$ but not $r_2$, then $G_1 \Box G_2$ is $\Z_{p_1}^{d_1} \times \Z_{p_2}^{d_2}$-distance antimagic.
\item[\rm (b)] If both $p_1$ and $p_2$ divide each of $r_1$ and $r_2$, then $G_1 \Box G_2$ is $\Z_{p_1}^{d_1} \times \Z_{p_2}^{d_2}$-distance magic.
\end{itemize}
\end{cor}

Theorem \ref{ham} and part (a) of Theorem \ref{thm:cart-prod-gen} together imply the following result.

\begin{cor}
\label{cor:cart-prod-1}
Let $d_i \ge 1$ and $q_i \ge 2$ be integers which are coprime, $1 \le i \le k$. If $q_i^{d_i}$ is a divisor of $\sum_{j \ne i} d_j (q_j - 1)$ for $1 \le i \le k$, then $H(d_1, q_1) \Box \cdots \Box H(d_k, q_k)$ is $\Z_{q_1^{d_1}} \times \cdots \times \Z_{q_k^{d_k}}$-distance antimagic.
\end{cor}

We will use a special case of this corollary, with all $(d_i, q_i) = (1, 2)$, in the proof of the following result.

\begin{thm}
\label{thm:hypercube}
Let $d \ge 3$ be an integer. If $d$ is odd or $d\equiv0\pmod4$, then $Q_d$ is $\Z_{2}^{d}$-distance antimagic.
\end{thm}
 
\begin{proof}
Choosing all $d_i = 1$ and $q_i = 2$, we have $Q_k = H(d_1, q_1) \Box \cdots \Box H(d_k, q_k)$, and $q_i^{d_i}$ divides $\sum_{j \ne i} d_j (q_j - 1)$ if and only if $k$ is odd. Thus, by Corollary \ref{cor:cart-prod-1}, if $k$ is odd, then $Q_k$ is $\Z_{2}^{k}$-distance antimagic.

Next we prove that $Q_4$ is $\Z_{2}^{4}$-distance antimagic. In fact, $Q_4$ can be viewed as $C_4 \Box C_4$, and one can verify that it admits a $\Z_{2}^{4}$-distance antimagic labelling as given in the following table:
\begin{center}
\begin{tabular}{|c|c|c|c|}
  \hline
  (0,0,0,0) & (0,1,0,0) & (1,1,1,1) & (1,1,0,0) \\
  (0,0,0,1) & (1,0,1,1) & (1,1,1,0) & (0,1,0,1) \\
  (0,0,1,1) & (1,0,1,0) & (1,0,0,1) & (0,1,1,0) \\
  (0,0,1,0) & (0,1,1,1) & (1,0,0,0) & (1,1,0,1) \\
  \hline
\end{tabular}
\end{center}
(In the table every entry represents a vertex label, and the vertex neighbours are the entries immediately above, below, on the left and on the right, where the top row is considered to be next to the bottom and the first column next to the last.)

Since $Q_4$ is $\Z_{2}^{4}$-distance antimagic, by part (a) of Corollary \ref{cor:cart-prod-2}, $Q_4 \Box Q_4$ is $\Z_{2}^{8}$-distance antimagic as $\Z_{2}^{4} \times \Z_{2}^{4} \cong \Z_{2}^{8}$. Similarly, by induction and part (a) of Corollary \ref{cor:cart-prod-2}, one can see that $Q_{4t} \cong Q_4 \Box \cdots \Box Q_4$ ($t$ factors) is $\Z_{2}^{4t}$-distance antimagic. In other words, if $k \ge 4$ is a multiple of $4$, then $Q_{k}$ is $\Z_{2}^{k}$-distance antimagic.  
\end{proof}
 
\begin{rem}
\label{rem:Q6}
If $Q_k$ is $\Z_{2}^{k}$-distance antimagic for some integer $k \ge 6$ with $k \equiv 2 \pmod4$, then $Q_d$ is $\Z_{2}^{d}$-distance antimagic for every integer $d \geq k$ with $d\equiv2\pmod4$. 

In particular, if we can prove that $Q_6$ is $\Z_{2}^{6}$-distance antimagic, then by Theorem \ref{thm:hypercube}, $Q_d$ is $\Z_{2}^{d}$-distance antimagic for all integers $d \ge 3$.
\end{rem}

In fact, we have $d = 4t + k$ for some $t$ and so $Q_d \cong Q_{4t} \Box Q_k$. We may assume $t > 0$. Since $Q_{4t}$ is $\Z_{2}^{4t}$-distance antimagic by Theorem \ref{thm:hypercube}, if $Q_k$ is $\Z_{2}^{k}$-distance antimagic, then by part (a) of Corollary \ref{cor:cart-prod-2}, $Q_d$ is $\Z_{2}^{d}$-distance antimagic.

Unfortunately, at the time of writing we do not know whether $Q_6$ is $\Z_{2}^{6}$-distance antimagic.  

Applying part (a) of Corollary \ref{thm:cart-prod} to cycles, we obtain the following result, where $C_n$ denotes the cycle of length $n$.

\begin{thm}
\label{cor:prod-cycle}
\begin{itemize}
\item[\rm (a)] If $n \ge 3$ is an odd integer, then $C_n$ is $\Z_n$-distance antimagic.
\item[\rm (b)] Let $n_1, \ldots, n_k \ge 3$ be odd integers, where $k \ge 2$. If each $n_i$ is a divisor of $k-1$, then $C_{n_1} \Box \cdots \Box C_{n_k}$ is $\Z_{n_1} \times \cdots \times \Z_{n_k}$-distance antimagic.
In particular, for any prime $p$ and any integer $d \ge 1$, $C_{p} \Box \cdots \Box C_{p}$ ($pd+1$ factors) is $\Z_p^{pd+1}$-distance antimagic.
\end{itemize}
\end{thm}

\begin{proof}
(a) Label the vertices along $C_n$ by $0, 1, \ldots, n-1$ consecutively. Since $n$ is odd, one can verify that this is a $\Z_n$-distance antimagic labelling of $C_n$.

(b) By (a), each $C_{n_i}$ is $\Z_{n_i}$-distance antimagic as $n_i$ is odd. Applying Corollary \ref{thm:cart-prod} and noting that all $d_i = 1$ and $r_i = 2$, we have $r - r_i = 2(k-1)$ and so the result follows.
\end{proof}

In part (b) of Theorem \ref{cor:prod-cycle}, $k$ ought to be relatively large comparable to $n_1, \ldots, n_k$. It would be interesting to find other conditions under which the same result holds.

We now give a generalisation of part (a) of Theorem \ref{thm:cart-prod-gen}. To this end we introduce the following concept which will also be used in the next section. Given an $r$-regular graph $G$ with $n$ vertices and an Abelian group $A$ of order $n$, a bijection $f: V(G) \rightarrow A$ is called an \emph{$A$-balanced labelling} if $w(x) = r f(x)$ for every $x \in V(G)$.  

\begin{thm}
\label{thm:cart-prod-gen1}
Let $G_i$ be an $r_i$-regular graph with $n_i \ge 2$ vertices, $1 \le i \le k$. Let $A_i$ be an Abelian group of order $n_i$, $1 \le i \le k$, and let $r = \sum_{i=1}^k r_i$. Suppose $\{1, \ldots, k\}$ is partitioned into, say, $I = \{1, \ldots, k_1\}$, $J = \{k_1 + 1, \ldots, k_2\}$ and $L = \{k_2 + 1, \ldots, k\}$ for some $0 \le k_1 \le k_2 \le k$, possibly with empty parts, such that the following conditions are satisfied:
\begin{itemize}
\item[\rm (a)] for each $i \in I$, $G_i$ is $A_i$-distance magic and $r-r_i$ is coprime to $n_i$; \item[\rm (b)] for each $j \in J$, $G_j$ is $A_j$-distance antimagic and $\exp(A_j)$ is a divisor of $r-r_j$;
\item[\rm (c)] for each $l \in L$, $G_l$ admits an $A_l$-balanced labelling and $r$ is coprime to $n_l$.
\end{itemize}
Then $G_1 \Box \cdots \Box G_k$ is $A_1 \times \cdots \times A_k$-distance antimagic.
\end{thm}

\begin{proof}
For $i \in I$, let $g_i$ be an $A_i$-distance magic labelling of $G_i$, with corresponding magic constant $\mu_i$. For $j \in J$, let $g_j$ be an $A_j$-distance antimagic labelling of $G_j$. For $l \in L$, let $g_l$ be an $A_l$-balanced labelling of $G_l$. Then, for every vertex $x = (x_1, \ldots, x_k)$ of $G_1 \Box \cdots \Box G_k$, we have $w_{G_i}(x_i) = \mu_i$ if $i \in I$, $(r-r_j)g_{j}(x_j) = 0$ if $j \in J$ as $\exp(A_j)$ is a divisor of $r-r_j$, and $w_{G_l}(x_l) = r_l g_{l}(x_l)$ if $l \in L$, where $w_{G_t}(x_t)$ is the weight of $x_t$ with respect to $g_t$. Define $f$ as in (\ref{eq:f}). Then the weight of $x$ under $f$ is given by (\ref{eq:w}). Thus the $t$th component of $w(x)$ is equal to $\mu_t + (r-r_t)g_{t}(x_t)$ if $t \in I$, $w_{G_{t}}(x_{t})$ if $t \in J$, and $rg_{t}(x_t)$ if $t \in L$.
For any vertex $x' = (x'_1, \ldots, x'_k) \ne x$ of $G_1 \Box \cdots \Box G_k$, there exists at least one $t$ such that $x_t \ne x'_t$. If $t \in I$, then $(r-r_t)g_{t}(x_t) \ne (r-r_t)g_{t}(x'_t)$ since $r-r_t$ is coprime to $n_t$ and the order of $g_{t}(x_t) - g_{t}(x'_t)$ ($\ne 0$) in $A_t$ is a divisor of $n_t$. If $t \in J$, then
$w_{G_t}(x_t) \ne w_{G_t}(x'_t)$ as $g_t$ is $A_t$-distance antimagic. If $t \in L$, then $rg_{t}(x_t) \ne rg_{t}(x'_t)$ as $r$ is coprime to $n_t$. In any case we have $w(x) \ne w(x')$. Therefore, $f$ is an $A_1 \times \cdots \times A_k$-distance antimagic labelling of $G_1 \Box \cdots \Box G_k$.
\end{proof}

Part (a) of Theorem \ref{thm:cart-prod-gen} can be obtained from Theorem \ref{thm:cart-prod-gen1} by setting $I = L = \emptyset$.

Theorem \ref{thm:cart-prod-gen1} enables us to construct new families of distance antimagic graphs based on known distance antimagic graphs and distance magic graphs. We illustrate this by the following corollary obtained by setting $n_t = 2^{d_t}$ and $A_t = \Z_2^{d_t}$ for $1 \le t \le s+t$, and $I = \{1, \ldots, s\}$, $J = \{s+1, \ldots, s+t\}$ and $L = \emptyset$ in Theorem \ref{thm:cart-prod-gen1}.

\begin{cor}
\label{cor:cart-prod-gen1}
Suppose $G_i$ is an $r_i$-regular $\Z_{2}^{d_i}$-distance magic graph, $1 \le i \le s$, and $G_j$ an $r_j$-regular $\Z_{2}^{d_j}$-distance antimagic graph, $s+1 \le j \le s+t$. Let $r = \sum_{i=1}^{s+t} r_i$ and $d = \sum_{i=1}^{s+t} d_i$. If $r_i$ and $r$ have different parity for $i = 1, \ldots, s$ and the same parity for $i = s+1, \ldots, s+t$, then $G_1 \Box \cdots \Box G_{s+t}$ is $\Z_2^d$-distance antimagic.

In particular, if $G_1$ is a regular $\Z_{2}^{d_1}$-distance magic graph with even degree, and $G_2$ a regular $\Z_{2}^{d_2}$-distance antimagic graph with odd degree, then $G_1 \Box G_{2}$ is $\Z_2^{d_1 + d_2}$-distance antimagic.
\end{cor}
 
The reader is invited to compare the last statement in Corollary \ref{cor:cart-prod-gen1} with Corollary \ref{cor:cart-prod-2}. To satisfy the conditions in Corollary \ref{cor:cart-prod-gen1}, $s$ must be even when $r$ is even, and $t$ must be odd when $r$ is odd.

\section{Labelling direct products}
\label{sec:dir}

The \emph{direct product} $G_1 \times \cdots \times G_d$ of $d$ graphs $G_1, \ldots, G_d$ is defined \cite{IK} to have vertex set $V(G_1) \times \cdots \times V(G_d)$ such that two vertices $(x_1, \ldots, x_d)$ and $(y_1, \cdots, y_d)$ are adjacent if and only if $x_i$ is adjacent to $y_i$ in $G_i$ for $i = 1, \ldots, d$.

\begin{thm}
\label{gr_pr}
Let $G_{i}$ be an $r_{i}$-regular graph with $n_i$ vertices and $A_i$ an Abelian group of order $n_i$, $i=1, \ldots, d$. Denote $r = r_1 \cdots r_d$. If $G_i$ is $A_{i}$-distance antimagic and $n_i$ and $r/r_i$ are coprime for $i=1, \ldots, d$, then $G_1 \times \cdots \times G_d$ is $A_{1} \times \cdots \times A_{d}$-distance antimagic.
\end{thm}

\begin{proof}
Since $G_i$ is $A_{i}$-distance antimagic, it admits at least one $A_{i}$-distance antimagic labelling, say, $g_{i}\colon V(G_{i}) \rightarrow A_{i}$, $i=1, \ldots, d$. Denote $G = G_1 \times \cdots \times G_d$ and $A = A_{1} \times \cdots \times A_{d}$. Define 
\begin{equation*}
f(x) = (g_1(x_1), \ldots, g_d(x_d)),\; \mbox{\rm for\; $x = (x_1, \ldots, x_d)$\; with $x_i \in V(G_i)$}.
\end{equation*}
Then $f$ is a bijection from $V(G)$ to $A$. The weight of $x$ under $f$ is given by
\begin{eqnarray*}
w(x) & = & \sum_{(y_1, \ldots, y_d) \in N_{G}(x)}(g_1(y_1), \ldots, g_d(y_d))\\[0.2cm]
&=& \left(\sum_{y_i \in N_{G_i}(x_i),\, i \ne 1}g_{1}(y_1), \ldots, \sum_{y_i \in N_{G_i}(x_i),\, i \ne d}g_{d}(y_d)\right) \\[0.2cm]
&=& ((r/r_1)w_{G_1}(x_1), \ldots, (r/r_d)w_{G_d}(x_d)).
\end{eqnarray*}%
Since $\gcd(r/r_i, n_i)=1$ for each $i$, the mapping defined by $(a_1, \ldots, a_d) \mapsto ((r/r_1)a_1, \ldots, (r/r_d)a_d)$, $(a_1, \ldots, a_d) \in A$, is a bijection from $A$ to itself. Combining this with the assumption that each $g_i$ is an $A_i$-distance antimagic labelling, we obtain that $f$ is an $A$-distance antimagic labelling of $G$.
\end{proof}

Denote $G^d_{\times} = G \times \cdots \times G$ ($d$ factors) and $A^d = A \times \cdots \times A$ ($d$ factors). Theorem \ref{gr_pr} implies:

\begin{cor}
Suppose $G$ is an $r$-regular $A$-distance antimagic graph with $n$ vertices, where $A$ is an Abelian group of order $n$. If $n$ and $r$ are coprime, then for any integer $d \ge 1$, $G^d_{\times}$ is $A^d$-distance antimagic.
\end{cor}

Since $C_n$ is $\Z_n$-distance antimagic when $n$ is odd (Theorem \ref{cor:prod-cycle}), Theorem \ref{gr_pr} implies:
\begin{cor}
For any odd integers $n_1, \ldots, n_d \ge 3$, $C_{n_1} \times \cdots \times C_{n_d}$ is $\Z_{n_1} \times \cdots \times \Z_{n_d}$-distance antimagic.

In particular, for any integer $d \ge 1$ and odd integer $n \ge 3$, $(C_n)_{\times}^d$ is $\Z_{n}^d$-distance antimagic.
\end{cor}

Since $K_n$ is $\Z_n$-distance antimagic (Theorem \ref{ham}), by Theorem \ref{gr_pr} we obtain:
\begin{cor}
Let $n_1, \ldots, n_d \ge 3$ be integers such that $n_i$ and $n_j - 1$ are coprime for distinct $i, j$. Then $K_{n_1} \times \cdots \times K_{n_d}$ is $\Z_{n_1} \times \cdots \times \Z_{n_d}$-distance antimagic.

In particular, for any integers $d \ge 1$ and $n \ge 3$, $(K_n)^d_{\times}$ is $\Z_{n}^d$-distance antimagic.
\end{cor}

Denote by $D_n=C_n\Box P_2$ the \emph{prism} of $2n \ge 6$ vertices.

\begin{lemma}
\label{obs:Cn-P2}
Let $n \geq 4$ be an integer not divisible by $3$. Then $D_n$ is $\Z_{2n}$-distance antimagic.
\end{lemma}

\begin{proof}
Denote the vertices of $D_n$ by $x_{i,j}$ such that $x_{i,j}x_{i+1,j}$ and $x_{i,1}x_{i,2}$ are edges of $D_n$ for $i = 0, 1, \ldots, n-1$ and $j = 1, 2$, where the first subscript is taken modulo $n$. Define
$$
f(x_{i,j})=\left\{
             \begin{array}{ll}
             2i, & \hbox { if $j=1$}\\ [0.2cm]
             2i+1, & \hbox{ if $j=2$}\\
             \end{array}
             \right.
$$
Then
$$
w(x_{i,1})= f(x_{i-1},1)+f(x_{i+1},1)+f(x_i,2) \equiv 6i+1 \pmod {2n}
$$
$$
w(x_{i,2})= f(x_{i-1},2)+f(x_{i+1},2)+f(x_i,1) \equiv 6i+2 \pmod {2n}.
$$
Since $n \neq 0 \pmod 3$, $g(x)=3x+1$, $x \in \Z_{2n}$ defines a bijection from $\Z_{2n}$ to itself. Therefore, the weights are all different elements of $\Z_{2n}$.
\end{proof}

Combining Theorem \ref{gr_pr} and Lemma \ref{obs:Cn-P2}, we obtain the following result.

\begin{cor}
Let $n_1, \ldots, n_d \ge 4$ be integers not divisible by $3$. Then $D_{n_1} \times \cdots \times D_{n_d}$ is $\Z_{n_1} \times \cdots \times \Z_{n_d}$-distance antimagic.

In particular, for any $d \ge 1$, and any $n \ge 4$ not divisible by $3$, $(D_n)_{\times}^d$ is $\Z_{n}^d$-distance antimagic.
\end{cor}

\section{$\Z_{n}$-distance antimagic product graphs}
\label{sec:Zn}

A \emph{balanced labelling} of an $r$-regular graph $G$ with $n$ vertices is defined as a $\Z_n$-balanced labelling of $G$; that is, a bijection $f: V(G) \rightarrow \Z_n$ such that $w(x) \equiv r f(x) \pmod n$ for every $x \in V(G)$.

\begin{lemma}
\label{obs:comp1}
\begin{itemize}
\item[\rm (a)] The cycle $C_n$ for any $n \ge 3$ admits a balanced labelling.
\item[\rm (b)] The complete graph $K_n$ on $n \geq 2$ vertices admits a balanced labelling if and only if $n$ is odd.
\end{itemize}
\end{lemma}

\begin{proof}
(a) The labelling that sequentially assigns $0, 1, \ldots, n-1$ to the vertices of $C_n$ along the cycle is balanced.

(b) Denote the vertices of $K_n$ by $x_i$, $1 \le i \le n$. Define $f: V(G) \rightarrow \Z_n$ by $f(x_i)=i-1$.
Then $(n-1)f(x_i) \equiv 1-i \pmod n$ for each $i$. If $n$ is odd, then $\sum_{i=1}^n f(x_i) \equiv n(n-1)/2 \equiv 0 \pmod n$ and the weight of $x_i$ is $w(x_i) = \sum_{j=1}^n f(x_j)-f(x_i) \equiv 1-i \equiv (n-1)f(x_i)$. Thus $f$ is a balanced labelling of $K_n$ when $n$ is odd.

Conversely, suppose $K_n$ admits a balanced labelling $f: V(G) \rightarrow \Z_n$. Then $w(x_i)=(n-1)f(x_i) \equiv -f(x_i) \pmod n$ for each $i$, and on the other hand $w(x_i)= \sum_{j=1}^n f(x_j)-f(x_i)$. Since $f$ is a bijection, we have
$$
\sum_{j=1}^n f(x_j) \equiv \sum_{j=1}^n j \equiv
\left\{
\begin{array}{ll}
0,               & \mbox{if } n \mbox{ is odd},\\ [0.2cm]
\frac{n}{2}, & \mbox{if } n \mbox{ is even}.
\end{array}
\right.
$$
It follows that $n$ must be odd.
\end{proof}

\begin{thm}
\label{thm:bal}
Suppose $G_i$ is an $r_i$-regular graph with $n_i$ vertices which admits a balanced labelling, $1 \le i \le k$. Suppose further that $r_i \le n_j$ for any $1 \le i, j \le k$ and that $r = \sum_{i=1}^k r_i$ is coprime to $n_1 \cdots n_k$. Then $G_1 \Box \cdots \Box G_k$ is $\Z_{n_1 \cdots n_k}$-distance antimagic.
\end{thm}

\begin{proof}
Let $g_i: V(G_i) \rightarrow \{0, 1, \ldots, n_i - 1\}$ be a balanced labelling of $G_i$, $1 \le i \le k$. Define  
\bea
f(x) & = & (g_1(x_1)\,\mbox{mod}\,{n_1}) + (g_2(x_2)\,\mbox{mod}\,{n_2})\ n_1 + (g_3(x_3)\,\mbox{mod}\,{n_3})\ n_1n_2 \non \\
& & \quad \quad + \cdots + (g_k(x_k)\,\mbox{mod}\,{n_k})\ n_1 \cdots n_{k-1} \label{eq:bal}
\eea
for $x = (x_1, \ldots, x_k) \in V(G_1) \times \cdots \times V(G_k)$. Then $f$ is a well-defined mapping from $V(G_1) \times \cdots \times V(G_k)$ to $\Z_{n_1 \cdots n_k}$ since it takes minimum value 0 and maximum value $(n_1 - 1) + (n_2 - 1)n_1 + \cdots + (n_k - 1)n_1 \cdots n_{k-1} = n_1 \cdots n_{k} - 1$. Moreover, $f$ is injective, and hence must be bijective since $|V(G_1) \times \cdots \times V(G_k)| = n_1 \cdots n_k$. In fact, if $f(x) = f(x')$ but $x \ne x'$, then $g_i(x_i) \ne g_i(x'_i)$ for at least one $i$ as all mappings $g_j$ are bijective. Let $t$ be the largest subscript such that $g_t(x_t) \ne g_t(x'_t)$. Then $g_1(x_1) + g_2(x_2)n_1 + \cdots + g_t(x_t)n_1 \cdots n_{t-1}  = g_1(x'_1) + g_2(x'_2)n_1 + \cdots + g_t(x'_t)n_1 \cdots n_{t-1}$. Without loss of generality we may assume $g_t(x_t) > g_t(x'_t)$. Then $(g_t(x_t) - g_t(x'_t)) n_1 \cdots n_{t-1} = (g_1(x'_1) - g_1(x_1)) + (g_2(x'_2) - g_2(x_2)) n_1 + \cdots + (g_{t-1}(x'_{t-1}) - g_{t-1}(x_{t-1})) n_1 \cdots n_{t-2}$. However, the right-hand side of this equality is no more than $(n_1 - 1) + (n_2 - 1) n_1 + \cdots + (n_{t-1} - 1) n_1 \cdots n_{t-2} = n_1 \cdots n_{t-1} - 1$, but the left-hand side of it is no less than $n_1 \cdots n_{t-1}$. This contradiction shows that $f$ is a bijection.

Denote $w_{G_i}(x_i) = \sum_{x_i':\, x_i x_i' \in E(G_i)} g_i(x'_i)$ for $x_i \in V(G_i)$, $1 \le i \le k$. With congruence modulo $n_1 \cdots n_k$, we have
\bean
w(x) & \equiv & \sum_{i=1}^k\; \sum_{x_i':\, x_i x_i' \in E(G_i)} f(x_1, \ldots, x_{i-1}, x'_i, x_{i+1}, \ldots, x_k) \\
& \equiv & \sum_{i=1}^k\; \sum_{x_i':\, x_i x_i' \in E(G_i)} \{(g_1(x_1)\,\mbox{mod}\,{n_1}) + \cdots + (g_i(x'_i)\,\mbox{mod}\,{n_i})\ n_1 \cdots n_{i-1} \\
& & \quad \quad + \cdots + (g_k(x_k)\,\mbox{mod}\,{n_k})\ n_1 \cdots n_{k-1}\} \\
& \equiv & \sum_{i=1}^k \{(r_i g_1(x_1)\,\mbox{mod}\,{n_1})\ + \cdots + (w_{G_i}(x_i)\,\mbox{mod}\,{n_i})\ n_1 \cdots n_{i-1} \\
& & \quad \quad + \cdots + (r_i g_k(x_k)\,\mbox{mod}\,{n_k})\ n_1 \cdots n_{k-1}\} \\
& \equiv & r f(x) + \sum_{i=1}^k \{(w_{G_i}(x_i) - r_i g_i(x_i))\,\mbox{mod}\,{n_i}\} n_1 \cdots n_{i-1} \\
& \equiv & r f(x).
\eean
In the last two steps we used the assumption that $r_i \le n_j$ for each pair $i, j$ and that $g_i$ is a balanced labelling of $G_i$. (Since $r_i \le n_j$, we have $(r_i g_j(x_j))\,\mbox{mod}\,{n_j} = r_i \cdot (g_j(x_j)\,\mbox{mod}\,{n_j})$.) If $w(x) \equiv w(x')$ (mod $n_1 \cdots n_k$), then $r(f(x) - f(x')) \equiv 0$ (mod $n_1 \cdots n_k$), which implies $f(x) \equiv f(x')$ since $r$ is coprime to $n_1 \cdots n_k$. Since $f$ is bijective as shown above, it follows that $f$ is a $\Z_{n_1 \cdots n_k}$-distance antimagic labelling of $G_1 \Box \cdots \Box G_k$.
\end{proof}

Denote $G^k_{\Box} = G \Box \cdots \Box G$ ($k$ factors). We have the following corollary of Theorem \ref{thm:bal} and Corollary \ref{obs:zn-even}.

\begin{cor}
\label{cor:bal}
Let $n_1, \ldots, n_k \ge 3$ be (not necessarily distinct) integers such that $k$ is coprime to $n_1 \cdots n_k$. Then $C_{n_1} \Box \cdots \Box C_{n_k}$ is $\Z_{n_1 \cdots n_k}$-distance antimagic if and only if all $n_1, \ldots, n_k$ are odd.

In particular, for any integer $k \ge 1$ and odd integer $n \ge 3$ that are coprime, $(C_n)_{\Box}^k$ is $\Z_{n^k}$-distance antimagic.
\end{cor}

\begin{proof}
Since $C_{n_1} \Box \cdots \Box C_{n_k}$ has degree $2k$, by Corollary \ref{obs:zn-even}, $C_{n_1} \Box \cdots \Box C_{n_k}$ cannot be $\Z_{n_1 \cdots n_k}$-distance antimagic unless $n_1, \ldots, n_k$ are all odd.

Suppose $n_1, \ldots, n_k \ge 3$ are all odd. Applying Theorem \ref{thm:bal} to $C_{n_1}, \ldots, C_{n_k}$, we have $r_i = 2 < n_j$, $r = 2k$, and each $C_{n_i}$ admits a balanced labelling by Lemma \ref{obs:comp1}. Since $n_1, \ldots, n_k$ are odd and $k$ is coprime to $n_1 \cdots n_k$ by our assumption, $r$ is coprime to $n_1 \cdots n_k$. Thus, by Theorem \ref{thm:bal}, $C_{n_1} \Box \cdots \Box C_{n_k}$ is $\Z_{n_1 \cdots n_k}$-distance antimagic.
\end{proof}

In particular, when $k$ is a prime, Corollary \ref{cor:bal} yields the following result.

\begin{cor}
\label{cor:bal1}
Let $p$ be a prime, and let $n_1, \ldots, n_p \ge 3$ be (not necessarily distinct) integers none of which has $p$ as a factor. Then $C_{n_1} \Box \cdots \Box C_{n_p}$ is $\Z_{n_1 \cdots n_p}$-distance antimagic if and only if all $n_1, \ldots, n_p$ are odd.

In particular, for any integers $n_1, n_2 \ge 3$, $C_{n_1} \Box C_{n_2}$ is $\Z_{n_1 n_2}$-distance antimagic if and only if both $n_1$ and $n_2$ are odd.

Moreover, if a prime $p$ is not a divisor of an odd integer $n \ge 3$, then $(C_n)_{\Box}^p$ is $\Z_{n^p}$-distance antimagic.
\end{cor}

Applying Theorem \ref{thm:bal} to $d$ copies of $K_q$ and using Lemma \ref{obs:comp1}, we obtain that $H(d, q)$ ($\cong (K_q)_{\Box}^d$) is $\Z_{q^d}$-distance antimagic for any $d \ge 1$ and odd $q \ge 3$ that are coprime. It is interesting to note that this is also given in part (a) of Theorem \ref{ham} where $q$ is not required to be odd.

\begin{thm}
\label{thm:dir}
Let $G_{i}$ be an $r_{i}$-regular graph with $n_i$ vertices, $i=1, \ldots, k$. Suppose $n_i$ and $r/r_i$ are coprime for $i=1, \ldots, k$, where $r = r_1 \cdots r_k$.
\begin{itemize}
\item[\rm (a)] If each $G_i$ is $\Z_{n_i}$-distance antimagic, then $G_1 \times \cdots \times G_k$ is $\Z_{n_1 \ldots n_k}$-distance antimagic.
\item[\rm (b)] If each $G_i$ is $\Z_{n_i}$-distance magic, then $G_1 \times \cdots \times G_k$ is $\Z_{n_1 \ldots n_k}$-distance magic.
\end{itemize}
\end{thm}

\begin{proof}
Denote $G = G_1 \times \cdots \times G_k$ and $n = n_1 \ldots n_k$.

(a) Define $f$ as in (\ref{eq:bal}) with the understanding that each $g_i$ is a $\Z_{n_i}$-distance antimagic labelling of $G_i$. As shown in the proof of Theorem \ref{thm:bal}, $f$ is a bijection from $V(G)$ to $\Z_{n}$. The weight of $x$ under $f$ is given by
\bean
w(x) & \equiv & \sum_{(y_1, \ldots, y_k) \in N_{G}(x)} f(y_1, \ldots, y_k) \\
& \equiv & \sum_{(y_1, \ldots, y_k) \in N_{G}(x)}\, \left\{\sum_{i=1}^k\, (g_i(y_i)\,\mbox{mod}\,{n_i})\ n_1 \cdots n_{i-1}\right\} \\
& \equiv & \sum_{i=1}^k\, \left\{\sum_{(y_1, \ldots, y_k) \in N_{G}(x)}\, (g_i(y_i)\,\mbox{mod}\,{n_i})\ n_1 \cdots n_{i-1}\right\} \\
& \equiv & \sum_{i=1}^k\, \left\{\left(\frac{r}{r_i} w_{G_i}(x_i)\,\mbox{mod}\,{n_i}\right)\ n_1 \cdots n_{i-1}\right\},
\eean
where the congruence is modulo $n$. Similar to the proof after (\ref{eq:bal}), one can show that for $x = (x_1, \ldots, x_k)$, $x' = (x'_1, \ldots, x'_k) \in V(G)$, $w(x)  \equiv w(x')\,\pmod {n_1 \cdots n_k}$ if and only if $\frac{r}{r_i} w_{G_i}(x_i)\,\mbox{mod}\,{n_i} \equiv \frac{r}{r_i} w_{G_i}(x'_i)\,\mbox{mod}\,{n_i}$ for each $i$. Since $\gcd(r/r_i, n_i) = 1$, the latter holds if and only if $w_{G_i}(x_i)\,\mbox{mod}\,{n_i}$ $\equiv$ $w_{G_i}(x'_i)\,\mbox{mod}\,{n_i}$ for each $i$. However, since $g_i$ is a $\Z_{n_i}$-distance antimagic labelling, we have $w_{G_i}(x_i)\,\mbox{mod}\,{n_i} \not \equiv w_{G_i}(x'_i)\,\mbox{mod}\,{n_i}$ for $x_i \ne x'_i$. Therefore, $w(x)  \not \equiv w(x')\,\pmod {n_1 \cdots n_k}$ for distinct vertices $x, x'$ of $G$, and hence $f$ is a $\Z_n$-distance antimagic labelling of $G$.

(b) Define $f$ as above with the understanding that each $g_i$ is a $\Z_{n_i}$-distance magic labelling of $G_i$. Denote by $\mu_i$ the magic constant of $G_i$ with respect to $g_i$. The computation above yields
$$
w(x) \equiv \sum_{i=1}^k\, \left\{\left(\frac{r\mu_i}{r_i}\,\mbox{mod}\,{n_i}\right)\ n_1 \cdots n_{i-1}\right\}.
$$
Since this is independent of $x$, $f$ is a $\Z_n$-distance magic labelling of $G$.
\end{proof}

Combining part (a) of Theorem \ref{thm:dir}, Corollary \ref{obs:zn-even} and part (a) of Theorem \ref{cor:prod-cycle}, we obtain:

\begin{cor}
\label{cor:dir}
Let $n_1, \ldots, n_k \ge 3$ be (not necessarily distinct) integers. Then $C_{n_1} \times \cdots \times C_{n_k}$ is $\Z_{n_1 \cdots n_k}$-distance antimagic if and only if all $n_1, \ldots, n_k$ are odd.

In particular, for any integer $d \ge 1$ and odd integer $n \ge 3$, $(C_n)_{\times}^d$ is $\Z_{n^d}$-distance antimagic.
\end{cor}

The \emph{strong product} $G_1 \boxtimes \cdots \boxtimes G_d$ of $d$ graphs $G_1, \ldots, G_d$ is defined \cite{IK} to have vertex set $V(G_1) \times \cdots \times V(G_d)$, such that two vertices $(x_1, \ldots, x_d)$ and $(y_1, \cdots, y_d)$ are adjacent if and only if, for $i = 1, \ldots, d$, either $x_i$ is adjacent to $y_i$ in $G_i$ or $x_i = y_i$.

\begin{thm}
\label{thm:str}
Let $G_i$ be an $r_{i}$-regular graph with $n_i$ vertices, for $i=1, 2$.
\begin{itemize}
\item[\rm (a)] Suppose both $G_1$ and $G_2$ have balanced labellings and $r_1 r_2 + r_1 + r_2$ is coprime to $n_1 n_2$. Then $G_1 \boxtimes G_2$ is $\Z_{n_1 n_2}$-distance antimagic.
\item[\rm (b)] Suppose $G_1$ is $\Z_{n_1}$-distance magic, $G_2$ is $\Z_{n_2}$-distance magic, $r_1$ is coprime to $n_2$, and $r_2$ is coprime to $n_1$. Then $G_1 \boxtimes G_2$ is $\Z_{n_1 n_2}$-distance antimagic.
\end{itemize}
\end{thm}

\begin{proof}
Denote $G = G_1 \boxtimes G_2$ and $n = n_1 n_2$.

(a) Let $g_i: V(G_i) \rightarrow \Z_{n_i}$ be a balanced labelling of $G_i$, $i=1,2$. Define $f(x) = (g_1(x_1)\,\mbox{mod}\,{n_1}) + (g_2(x_2)\,\mbox{mod}\,{n_2}) n_1$ for $x=(x_1, x_2) \in V(G_1) \times V(G_2)$. 
Then $f$ is a bijection from $V(G)$ to $\Z_{n}$ as seen in the paragraph after (\ref{eq:bal}).
With congruence modulo $n$, we have
\bean
w(x) & \equiv & \sum_{(y_1, y_2) \in N_{G}(x)} f(y_1, y_2) \\
& \equiv & \sum_{y_1 \in N_{G_1}(x_1)} f(y_1, x_2) + \sum_{y_2 \in N_{G_2}(x_2)} f(x_1, y_2)  + \sum_{y_1 \in N_{G_1}(x_1),\, y_2 \in N_{G_2}(x_2)} f(y_1, y_2) \\
& \equiv & \sum_{y_1 \in N_{G_1}(x_1)} \left\{(g_1(y_1)\,\mbox{mod}\,{n_1}) + (g_2(x_2)\,\mbox{mod}\,{n_2})n_1\right\} \\
& & + \sum_{y_2 \in N_{G_2}(x_2)} \left\{(g_1(x_1)\,\mbox{mod}\,{n_1}) + (g_2(y_2)\,\mbox{mod}\,{n_2})n_1\right\} \\
& & + \sum_{y_1 \in N_{G_1}(x_1),\, y_2 \in N_{G_2}(x_2)} \left\{(g_1(y_1)\,\mbox{mod}\,{n_1}) + (g_2(y_2)\,\mbox{mod}\,{n_2})n_1\right\} \\
& \equiv & (w_{G_1}(x_1)\,\mbox{mod}\,{n_1}) + (r_1 g_2(x_2)\,\mbox{mod}\,{n_2})n_1 + (r_2 g_1(x_1)\,\mbox{mod}\,{n_1}) + (w_{G_2}(x_2)\,\mbox{mod}\,{n_2})n_1 \\
& & + (r_2 w_{G_1}(x_1)\,\mbox{mod}\,{n_1}) + (r_1 w_{G_2}(x_2)\,\mbox{mod}\,{n_2})n_1 \\
& \equiv & \{[(r_2 + 1) w_{G_1}(x_1) + r_2 g_1(x_1)]\,\mbox{mod}\,{n_1}\} + \{[(r_1 + 1) w_{G_2}(x_2) + r_1 g_2(x_2)]\,\mbox{mod}\,{n_2}\}n_1 \\
& \equiv & \{[(r_1 r_2 + r_1 + r_2) g_1(x_1)]\,\mbox{mod}\,{n_1}\} + \{[(r_1 r_2 + r_1 + r_2) g_2(x_2)]\,\mbox{mod}\,{n_2}\}n_1.
\eean
In the last step we used the assumption that $g_1, g_2$ are balanced labellings of $G_1, G_2$ respectively. Since $r_1 r_2 + r_1 + r_2$ is coprime to $n_1n_2$, the computation above implies $w(x) \ne w(x')$ for distinct $x, x' \in V(G)$. Thus $f$ is a $\Z_n$-distance antimagic labelling of $G$.

(b) Define $f: V(G) \rightarrow \Z_{n}$ as in (a) with the understanding that $g_i$ is a $\Z_{n_i}$-distance magic labelling of $G_i$, $i=1,2$. Let $\mu_1, \mu_2$ be the corresponding magic constants of $g_1, g_2$ respectively. The computation in (a) implies that, for $x = (x_1, x_2) \in V(G)$, $w(x) = \{[(r_2 + 1) \mu_1 + r_2 g_1(x_1)]\,\mbox{mod}\,{n_1}\} + \{[(r_1 + 1) \mu_2 + r_1 g_2(x_2)]\,\mbox{mod}\,{n_2}\}n_1$. Since both $g_1$ and $g_2$ are bijective, and since $r_1$ is coprime to $n_2$ and $r_2$ is coprime to $n_1$, one can show that $w(x) \ne w(x')$ for distinct $x, x' \in V(G)$. Therefore, $f$ is a $\Z_n$-distance antimagic labelling of $G$.
\end{proof}

Theorem \ref{thm:str} involves only two factor graphs. Nevertheless, with the help of the following lemma (together with the associativity \cite{IK} of the strong product), we can obtain results about strong products with more than two factors by recursively applying part (a) of Theorem \ref{thm:str}.

\begin{lemma}
\label{lem:bl}
Let $G_1$ and $G_2$ be regular graphs. If both $G_1$ and $G_2$ have balanced labellings, then $G_1 \boxtimes G_2$ has a balanced labelling.
\end{lemma}

\begin{proof}
Denote $n_i = |V(G_i)|$, and let $g_i: V(G_i) \rightarrow \Z_{n_i}$ be a balanced labelling of $G_i$,  $i=1,2$. Define $g: V(G_1) \times V(G_2) \rightarrow \Z_{n_1 n_2}$ by $g(x_1, x_2) = g_1(x_1) + g_2(x_2) n_1$. Then $g$ is a bijection from $V(G_1) \times V(G_2)$ to $\Z_{n_1 n_2}$. One can verify that $g$ is a balanced labelling of $G_1 \boxtimes G_2$.
\end{proof}

Since any cycle $C_n$ admits a $\Z_n$-balanced labelling, in view of Corollary \ref{obs:zn-even}, we obtain the following result from part (a) of Theorem \ref{thm:str}.

\begin{cor}
\label{cor:str}
Let $m, n \ge 3$ be (not necessarily distinct) integers. Then $C_{m} \boxtimes C_{n}$ is $\Z_{mn}$-distance antimagic if and only if both $m$ and $n$ are odd.
\end{cor}

Similarly, since $K_n$ admits a $\Z_n$-balanced labelling if $n$ is odd (Lemma \ref{obs:comp1}), by part (a) of Theorem \ref{thm:str} we obtain the following corollary.

\begin{cor}
\label{cor:boxcomplete}
For any odd integers $m, n \geq 3$ (not necessarily distinct), $K_{m} \boxtimes K_{n}$ is $\Z_{mn}$-distance antimagic.
\end{cor}

We conclude this section with the following results. 

\begin{thm}
\label{Knregular}
Let $G$ be a regular graph on $m$ vertices other than the empty graph $\overline{K}_m$.
\begin{itemize}
  \item[\rm (a)] If $n$ is even, then $K_n\boxtimes G$ is $\Z_{nm}$-distance antimagic.
  \item[\rm (b)] If both $n$ and $m$ are odd, then $K_n\boxtimes G$ is $\Z_{nm}$-distance antimagic.
\end{itemize}
\end{thm}

\begin{proof}
Let $r \ge 1$ be the degree of $G$. Then $K_n\boxtimes G$ is $(n(r+ 1)-1)$-regular with $mn$ vertices. Denote $V(K_n) = \{x_1,x_2,\ldots,x_n\}$ and $V(G) = \{v_1,v_2,\ldots,v_m\}$.

(a) Let $n$ be even. As in~\cite{Be}, define $f: V(K_n\boxtimes G)\rightarrow \{1,2,\ldots,mn\}$ by
$$
f(x_j,v_i)=\left\{
      \begin{array}{ll}
        \frac{i-1}{2}n +j-1, & \hbox{if } j \hbox{ is odd,} \\[0.2cm]
        mn-1-f(x_{j-1},v_i), &  \hbox{if } j \hbox{ is even}
      \end{array}
    \right.
$$
for $1 \leq i \leq m$ and $1 \leq j \leq n$. It can be verified that $f$ is a bijection and
$w(x)=(r+1)\frac{n}{2}-f(x)$ for each $x\in V(K_n\boxtimes G)$. Hence $f$ is a $\Z_{nm}$-distance antimagic labelling of $K_n\boxtimes G$.

(b) Suppose both $n$ and $m$ are odd. Then there exists \cite{Har1,Har2} a magic $(n, m)$-rectangle, say, with $(j,i)$-entry $a_{j,i}$, $1\leq j \leq n$, $1\leq i\leq m$ and $\sum_{j=1}^{n}a_{j,i}=C$ for some constant $C$ and every $i$. (A magic $(n,m)$-rectangle is an $n \times m$ array in which each of the integers $1, 2, \ldots, mn$ occurs exactly once such that the sum over each row is a constant and the sum over each column is also a constant.) 
Define $f \colon V(K_n \boxtimes G)\rightarrow \{1,2,\ldots,nm\}$ by $f(x_j,v_i)=a_{j,i}$ for each pair $(j, i)$. Obviously, $f$ is a bijection and moreover $\sum_{j=1}^{n}f(x_j,v_i)=C$ for every $i$. Therefore, for any $x\in V(K_n\boxtimes G)$, we have
$w(x)=(r+1)C-f(x)$. It follows that $f$ is a $\Z_{nm}$-distance antimagic labelling of $K_n\boxtimes G$.
\end{proof}

\section{$\Z_{n}$-distance antimagic circulants}
\label{sec:circulant}

A \emph{circulant graph} is a Cayley graph $\Cay(\Z_n, S)$ on a cyclic group $\Z_n$, where $S \subseteq \Z_n \setminus \{0\}$ such that $S = -S = \{-s: s \in S\}$ with operation modulo $n$. The degree $|S|$ of $\Cay(\Z_n, S)$ is odd if $n$ is even and $n/2 \in S$, and even otherwise. The purpose of this section is to show the following result on $\Z_{n}$-distance antimagic circulants.
 
\begin{thm}
\label{thm:circ1}
The circulant graph $\Cay(\Z_{n},S)$ is $\Z_{n}$-distance antimagic if one of the following holds:
\begin{enumerate}
\item[\rm (a)] $n$ is even, $n/2 \in S$, $S$ contains $2s$ even integers and $2t$ odd integers other than $n/2$, and $2(s-t)+1$ and $2(t-s)+1$ are both coprime to $n$;
\item[\rm (b)] $|S|$ and $n$ are coprime.
\end{enumerate}
 
On the other hand, if $|S|$ and $n$ are both even, then $\Cay(\Z_{n},S)$ is not $\Z_{n}$-distance antimagic. 
\end{thm}

\begin{proof}
Part (b) follows from Theorem \ref{obs:cay}. To prove (a), assume $n=2m$ is even and $m \in S$. Denote $\Z_n = \{0, 1, \ldots, 2m-1\}$ and let $S=\{\pm i_1,\pm i_2,\dots,\pm i_s,\pm j_1,\pm j_2,\dots,\pm j_t,m\}$, with operation modulo $2m$, where $0<i_1<i_2<\dots<i_s~(<m)$ are even and $0<j_1<j_2<\dots<j_t~(<m)$ are odd. Define the labelling as $f(k)=k$ for $k$ odd and $f(k)=-k$ for $k$ even. If $k$ is even, then the weight of vertex $k$ under $f$ is given by
\begin{align*}
w(k)
= & \sum_{\ell =1}^s (f(k+i_{\ell})+f(k-i_{\ell})) + \sum_{\ell=1}^t (f(k+j_{\ell})+f(k-j_{\ell}))+f(k+m)\\
= & \sum_{\ell =1}^s (-(k+i_{\ell})-(k-i_{\ell})) + \sum_{\ell=1}^t ((k+j_{\ell})+(k-j_{\ell}))+f(k+m)\\
=&2(t-s)k + f(k + m).
\end{align*}
Thus $w(k) = (2(t-s)-1)k - m$ if $m$ is even, and $w(k) = (2(t-s)+1)k + m$ if $m$ is odd. 
Since both $2(s-t)+1$ and $2(t-s)+1$ are coprime to $2m$, if $m$ is even, then $w$ induces a bijection from even elements of $\Z_{2m}$ to even elements of $\Z_{2m}$; and if $m$ is odd, then $w$ induces a bijection from even elements of $\Z_{2m}$ to odd elements of $\Z_{2m}$.
 
Similarly, when $k$ is odd, we have
\begin{align*}
w(k)
= & \sum_{\ell =1}^s ((k+i_{\ell})+(k-i_{\ell})) + \sum_{\ell=1}^t (-(k+j_{\ell})-(k-j_{\ell}))+f(k+m)\\
= & 2(s-t)k + f(k + m).
\end{align*}
Thus $w(k) = (2(s-t)+1)k + m$ if $m$ is even, and $w(k) = (2(s-t)-1)k - m$ if $m$ is odd. Similar to the above, if $m$ is even, then $w$ induces a bijection from odd elements of $\Z_{2m}$ to odd elements of $\Z_{2m}$; and if $m$ is odd, then $w$ induces a bijection from odd elements of $\Z_{2m}$ to even elements of $\Z_{2m}$. Combining this with what we proved in the previous paragraph, we conclude that $w$ is a bijection from $\Z_{2m}$ to $\Z_{2m}$, and this proves part (a).

By Corollary \ref{obs:zn-even}, if both $n$ and $|S|$ are even, then $\Cay(\Z_{n},S)$ is not $\Z_n$-distance antimagic.
\end{proof}

\section{Remarks and questions}

Many problems and questions arise from our studies. We list a few of them with no attempt to be exhaustive.
\begin{enumerate}
\item When does a Cayley graph on $\Z_2^d$ with even degree (odd degree, respectively) admit a $\Z_2^d$-distance antimagic (magic, respectively) labelling? (See Corollary \ref{cor:cay1}.)
\item Give a necessary and sufficient condition for the Hamming graph $H(d, q)$ ($d \ge 1, q \ge 3$) to be $\Z_{q^d}$-distance antimagic. (Partial results were obtained in Theorem \ref{ham}, and when $q=2$ the answer was given in Corollary \ref{thm:cube-Zn}.)
\item Prove or disprove that $Q_6$ is $\Z_2^6$-distance antimagic. (See Remark \ref{rem:Q6}.) 
\item Give a necessary and sufficient condition for $C_{n_1} \Box \cdots \Box C_{n_d}$ to be $\Z_{n_1 \cdots n_d}$-distance antimagic when $d$ is not coprime to $n_1 \cdots n_d$. (A necessary condition is that all $n_i$'s must be odd. See Corollary \ref{cor:bal} for the case when $d$ and $n_1 \cdots n_d$ are coprime.)
\item Give a necessary and sufficient condition for a circulant graph on $n$ vertices to be $\Z_n$-distance antimagic. (See Section \ref{sec:circulant}.)
\end{enumerate}

\medskip

\noindent \textbf{Acknowledgements}~~Cichacz was partially supported by the Polish Ministry of Science and Higher Education. Zhou was supported by the Australian Research Council (FT110100629).

\small{

}


\begin{thebibliography}{99}

\bibitem{afk}
S. Arumugam, D. Froncek and N. Kamatchi, Distance magic graphs -- a survey, {\it J. Indones. Math. Soc.}, Special Edition (2011), 1--9.

\bibitem{BM}
M. Baca and M. Miller, Super Edge-Antimagic Graphs: A Wealth of Problems and Some Solutions, Brown Walker Press, Boca Raton, 2008.

\bibitem{Be}
S. Beena, On $\Sigma$ and $\Sigma'$ labelled graphs, {\it Discrete Math.} 309 (2009) 1783--1787.

\bibitem{CCN}
N. Cavenagh, D. Combe and A. M. Nelson, Edge-magic group labellings of countable graphs, {\it Electron. J. Combin.} 13 (2006), \# R92.


\bibitem{CiFrGrKo}
D. Froncek, P. Gregor, P. Kovar, Distance magic hypercubes, manuscript.

\bibitem{ComNelPal}
D. Combe, A.M. Nelson and W.D. Palmer, Magic labellings of graphs over
finite abelian groups, {\it Australas. J. Combin.} 29 (2004) 259--271.

\bibitem{Fr}
D. Froncek, {Group distance magic labeling of Cartesian products of cycles}, {\it Australas. J. Combin.} 55 (2013) 167--174.

\bibitem{gal}
J. Gallian, A Dynamic Survey of Graph Labeling, {\it Electron. J. Combin.} 16  (2009), \# DS6.

\bibitem{Har1}
T. Harmuth, \"{U}ber magische Quadrate und \"{a}hniche Zahlenfiguren, {\it Arch. Math. Phys.} 66 (1881) 286--313.

\bibitem{Har2}
T. Harmuth, \"{U}ber magische Rechtecke mit ungeraden Seitenzahlen, {\it Arch. Math. Phys.} 66 (1881) 413--447.

\bibitem{IK}
W. Imrich and S. Klav\v{z}ar, Product Graphs: Structure and
Recognition, John Wiley \& Sons, New York, 2000.

\bibitem{LL}
R. M. Low and S-M. Lee, On the products of groupmagic graphs, {\it Australas. J. Combin.} 34 (2006), 41--48. 

\bibitem{LS}
R. M. Low and L. Sue, Some new results on the integer-magic spectra of tessellation graphs,
{\it Australas. J. Combin.} 38 (2007), 255--266. 

\bibitem{rao}
S. B. Rao, T. Singh and V. Parameswaran, Some sigma labelled graphs I, In: Graphs, Combinatorics, Algorithms and Applications, eds. S. Arumugam, B.D. Acharya and S.B. Rao, Narosa Publishing House, New Delhi, (2004), 125--133.

\bibitem{SLS}
W. C. Shiu, P. C. B. Lam and P. K. Sun, Construction of group-magic graphs and some $A$-magic graphs with $A$ of even order, In: Proceedings of the Thirty-Fifth Southeastern International Conference on Combinatorics, Graph Theory and Computing, 
{\it Congr. Numer.} 167 (2004), 97--107. 

\bibitem{Scott}
W. R. Scott, Group Theory, Prentice-Hall, Englewood Cliffs, N. J., 1964.

\bibitem{wal}
W. D. Wallis, Magic Graphs, Birkh\"auser, Boston-Basel-Berlin, 2001.

\end{thebibliography}
\end{document}